\documentclass[a4paper,11pt]{article}
\usepackage{amsmath,amsthm,amssymb}
\usepackage{hyperref}
\usepackage[T1]{fontenc}

\parskip \medskipamount

\title{A Note on Hamiltonian-Intersecting Families of Graphs}
\author{Imre Leader\thanks{Department of Pure Mathematics and Mathematical Statistics, Centre for Mathematical Sciences, University of Cambridge, Wilberforce Road, Cambridge CB3 0WB, United Kingdom. Email: \texttt{i.leader@dpmms.cam.ac.uk}.} \and  \v{Z}arko Ran\dj elovi\'c\thanks{Department of Pure Mathematics and Mathematical Statistics, Centre for Mathematical Sciences, University of Cambridge, Wilberforce Road, Cambridge CB3 0WB, United Kingdom. Email: \texttt{zr233@cam.ac.uk}.} \and Ta Sheng Tan\thanks{Institute of Mathematical Sciences, Faculty of Science, Universiti Malaya, 50603 Kuala Lumpur, Malaysia. Email: \texttt{tstan@um.edu.my}.}}

\newtheorem{thm}{Theorem}

\newtheorem{conjecture}[thm]{Conjecture}

\theoremstyle{remark}  

\theoremstyle{definition}

\begin{document}

\maketitle
\begin{abstract}
 How many graphs on an $n$-point set can we find such that any two have connected intersection?
 Berger, Berkowitz, Devlin, Doppelt, Durham, Murthy and Vemuri showed that the maximum is exactly $1/2^{n-1}$ of all graphs.
 Our aim in this short note is to give a `directed' version of this result;
 we show that a family of oriented graphs such that any two have strongly-connected intersection has size at most $1/3^n$ of all oriented graphs.
 We also show that a family of graphs such that any two have Hamiltonian intersection has size at most $1/2^n$ of all graphs, verifying a conjecture of the above authors.
\end{abstract}


\section{Introduction}
For a graph property $\mathcal{P}$, we say that a family $\mathcal{F}$ of graphs on a common vertex set $[n]=\{1,2,\ldots,n\}$ is \emph{$\mathcal{P}$-intersecting} if $G\cap H$ satisfies $\mathcal{P}$ for every $G$ and $H$ in $\mathcal{F}$.
The largest size of a $\mathcal{P}$-intersecting family has been investigated for many properties $\mathcal{P}$ by several authors: see the survey of Ellis~\cite{ellis} for many examples.

For the property of connectedness (meaning being a connected spanning subgraph of $[n]$), one obvious family is the family of all graphs containing a fixed spanning tree $T$.
This family has size $2^{\binom{n}{2}-(n-1)}$, or equivalently, it has size $1/2^{n-1}$ of all graphs.
Berger, Berkowitz, Devlin, Doppelt, Durham, Murthy and Vemuri~\cite{berger1} showed that this is in fact the maximum size.

For completeness, we give their (very short and elegant) proof.
Viewing the space of graphs as an $\binom{n}{2}$-dimensional vector space over $\mathbb{Z}_2$, consider the linear span $S$ of the stars $S_1, S_2, \ldots, S_{n-1}$ centred at $1,2,\ldots, n-1$, respectively.
Then $S$ has dimension $n-1$ (as those stars are linearly independent).
Also, every non-empty sum of the $S_i$ is an edge cut of the complete graph, and so has disconnected complement.
It follows that if $\mathcal{F}$ is a connected-intersecting family of graphs, then $\mathcal{F}$ can meet each translate of $S$ in at most one point, since for any graph $G$ and any non-zero $H\in S$, the graphs $G$ and $G+H$ have intersection containing no edge of $H$.
This completes the proof.

What if we ask that any two graphs from our family have intersection that contain a Hamilton cycle?
The authors of \cite{berger1} conjectured that the greatest size of a Hamiltonian-intersecting family is attained by the family of all graphs containing a fixed Hamilton cycle.
In this short note, we point out that (perhaps surprisingly) this may in fact be deduced in a direct way from the connected-intersecting result itself.

Interestingly, our proof also shows uniqueness: the only Hamiltonian-intersecting families of size $1/2^n$ of all graphs are the families of the above form.
This is in contrast with the connectedness case, where even for Hamilton-path-intersecting we do not have uniqueness.
Indeed, as noted in \cite{berger1}, if we consider the family $\mathcal{F}$ consisting of all graphs that contain the path $34\cdots n$, and the edge $12$, and at least two of the edges $13, 23$ and $1n$, we see that $\mathcal{F}$ is Hamilton-path-intersecting and has size $1/2^{n-1}$ of all graphs.

What happens for directed graphs?
Here the natural analogue of connectedness is strong connectedness: for any $x$ and $y$, there is a directed path from $x$ to $y$.
For \emph{directed} graphs (meaning we allow both $\overrightarrow{xy}$ and
$\overrightarrow{yx}$), the natural guess is that the maximum size is $1/2^n$ of all directed graphs, obtained by taking all directed graphs that contain a fixed directed $n$-cycle.
This is actually not so interesting: it turns out to follow from linear algebra arguments exactly as above.
The more interesting question is for \emph{oriented} graphs (meaning we cannot have both $\overrightarrow{xy}$ and $\overrightarrow{yx}$).
Here the same construction has size $1/3^n$ of all oriented graphs.
We prove that this is indeed best.

Returning to undirected graphs, we mention that there are several related problems that we are unable to solve.
We make some conjectures at the end of this note.

\section{Proofs of Results}

We start with the simplest result, for directed graphs.

\begin{thm}\label{directed}
 Let $\mathcal{F}$ be a family of directed graphs on $[n]$ such that the intersection of any two graphs in $\mathcal{F}$ is strongly connected.
 Then $|\mathcal{F}|\le \frac{1}{2^n}\cdot 4^{\binom{n}{2}}$.
\end{thm}
\begin{proof}
 We consider the space of all directed graphs as a vector space of dimension $2\binom{n}{2}$ over $\mathbb{Z}_2$.
 For $1\le i\le n$, let $S_i$ be the out-star centred at $i$: $S_i$ = $\{\overrightarrow{ij}:j\ne i\}$.
 Then the $S_i$ are linearly independent (as they are disjointly supported), so their span $S$ has dimension $n$.
 To show that $\mathcal{F}$ meets each translate of $S$ in at most one point, it suffices to show that every non-zero graph in $S$ contains all directed edges from $A$ to $A^C$ for some $A\subset [n]$ ($A\ne \emptyset, [n]$), because every strongly connected graph must contain such a directed edge.
 Indeed, if both of the graphs $G$ and $G+H$ belonged to $\mathcal{F}$ for some non-zero $H\in S$, then $G\cap (G+H)$ contains no directed edge in $H$, contradicting the fact that $G\cap(G+H)$ is strongly connected.

 But a non-zero sum of the generators of $S$, say involving $S_i$, already must contain this cut for $A=\{i\}$.
 We conclude that the size of $\mathcal{F}$ is at most $1/2^n$ of the total number of directed graphs on $[n]$.
\end{proof}

We now move on to the case of oriented graphs.
Here, we cannot see a linear algebra approach.
Of course, we may view the space of all oriented graphs as an $\binom{n}{2}$-dimensional vector space over $\mathbb{Z}_3$ in the obvious way (with $\overrightarrow{xy}+\overrightarrow{yx}=0$ and $\overrightarrow{xy}+\overrightarrow{xy}=\overrightarrow{yx}$).
But there does not seem to be an $n$-dimensional subspace $S$ such that every non-zero graph in $S$ contains a directed cut as above.

Instead, we will use projection (entropy), using Shearer's lemma~\cite{chung} or the Uniform Covers theorem of Bollob\'as and Thomason~\cite{bollobas} by considering projections onto stars.

\begin{thm}\label{oriented}
 Let $\mathcal{F}$ be a family of oriented graphs on $[n]$ such that the intersection of any two graphs in $\mathcal{F}$ is strongly connected.
 Then $|\mathcal{F}|\le \frac{1}{3^n}\cdot 3^{\binom{n}{2}}$.
\end{thm}
\begin{proof}
 We view an oriented graph $G$ as a point $\Tilde{G}$ in $\{0,1,2\}^{\binom{n}{2}}$ (where the coordinates are indexed by edges of the complete graph), by giving the $xy$-coordinate value $0$ if the edge $xy$ is not present in either direction, and value $1$ if it is present in a fixed reference direction (e.g., $\overrightarrow{xy}$ where $x<y$), and value $2$ if it is present in the other direction. 
 Our family $\mathcal{F}$ thus corresponds to a set $\Tilde{\mathcal{F}}\subset \{0,1,2\}^{\binom{n}{2}}$.

 For each $i$, consider the projection $\pi$ onto the $n-1$ coordinates involving $i$.
 For any graphs $G$ and $H$ in $\mathcal{F}$, $G\cap H$ contains at least one edge out of $i$ and at least one edge into $i$.
 Hence $\pi(\Tilde{G})$ and $\pi(\Tilde{H})$ agree on at least two coordinates. (In fact, they are non-zero on these coordinates, but we will not use this fact.)

 So, how many points in the set $\{0,1,2\}^{n-1}$ can we find such that any two agree on at least two coordinates?
 This question was answered by Frankl and Tokushige~\cite{frankl}, who showed that this maximum is $3^{n-3}$. (See Section 3 of their paper for a particularly attractive proof of this result; they also found the complete answer for more general alphabet size and for agreement in $t$ coordinates.)
 It follows that $\pi(\Tilde{\mathcal{F}})$ has size at most $3^{n-3}$, which is $3^{(n-1)\alpha}$ where $\alpha = 1-\frac{2}{n-1}$.

 Now, the stars centred at $i$, for $1\le i\le n$, form a `uniform cover' of the complete graph, in the sense that each edge of the complete graph belongs to the same number of stars.
 Hence, by Shearer's lemma or the Uniform Covers theorem, it follows that the size of $\Tilde{\mathcal{F}}\subset \{0,1,2\}^{\binom{n}{2}}$ is at most $3^{\binom{n}{2}\alpha}$.
 And this is exactly $3^{\binom{n}{2}-n}$, as required.
\end{proof}

We remark that equality occurs in Theorem~\ref{oriented} if and only if $\mathcal{F}$ consists of all oriented graphs that contain a fixed directed Hamilton cycle.
Indeed, by the case of equality in the Uniform Covers theorem (see \cite{bollobas}), for equality to hold in Theorem~\ref{oriented} we must have that $\Tilde{\mathcal{F}}$ is a cuboid, that is, a product of $1$-dimensional sets.
This cuboid must have $n$ side-lengths of $1$, and the remaining $\binom{n}{2}-n$ side-lengths of $3$. 
This corresponds to $\mathcal{F}$ consisting of all oriented graphs that contain some fixed $n$-edge oriented graph that is strongly connected, as required.

We now turn to undirected graphs, and the question of Hamiltonian-intersecting families.
One could try to project onto stars, as for the oriented case.
On each star, our family will be $2$-intersecting.
However, a $2$-intersecting family of sets can have size about $1/2$ of $2^n$, by taking for example the family of all sets of size at least $\frac{n}{2}+1$ -- in complete contrast to the case of $2$-intersecting families in $\{0,1,2\}^n$, as discussed above.
So the bound we get would be incredibly weak.

Remarkably, though, if we instead project onto complete graphs, then the bounds work out perfectly.

\begin{thm}\label{hamiltonian}
 Let $\mathcal{F}$ be a family of graphs on $[n]$ such that the intersection of any two graphs in $\mathcal{F}$ is Hamiltonian.
 Then $|\mathcal{F}|\le \frac{1}{2^n}\cdot 2^{\binom{n}{2}}$.
\end{thm}
\begin{proof}
 We view a graph $G$ as a point $\Tilde{G}$ in $\{0,1\}^{\binom{n}{2}}$ in the obvious way.
 Our family $\mathcal{F}$ thus corresponds to a set $\Tilde{\mathcal{F}}\subset \{0,1\}^{\binom{n}{2}}$.

 For each $i$, consider the projection $\pi$ onto the $\binom{n-1}{2}$ coordinates not involving $i$.
 For any graphs $G$ and $H$ in $\mathcal{F}$, $G\cap H$ contains a Hamilton path on $[n]\setminus\{i\}$, and in particular, the induced subgraph of $G\cap H$ on $[n]\setminus\{i\}$ is connected.
 Hence $\pi(\Tilde{\mathcal{F}})$ is connected-intersecting on $[n]\setminus\{i\}$.
 It follows from the connected-intersecting result in \cite{berger1} that the size of $\pi(\Tilde{\mathcal{F}})$ is at most $2^{\binom{n-1}{2}-(n-2)}$, which is $2^{\binom{n-1}{2}\alpha}$ where $\alpha=1-\frac{2}{n-1}$.

 Now, the $n$ complete graphs on $n-1$ vertices form a uniform cover of the complete graph on $n$ vertices.
 Hence as before it follows that the size of $\Tilde{\mathcal{F}}$ is at most $2^{\binom{n}{2}\alpha}=2^{\binom{n}{2}-n}$.
\end{proof}

We remark that, as before, from the equality case in the Uniform Covers theorem it follows that equality holds in Theorem~\ref{hamiltonian} if and only if $\mathcal{F}$ consists of all graphs that contain a fixed Hamilton cycle.

\section{Open Problems}

Theorem~\ref{hamiltonian} is stated for Hamiltonian-intersecting families, but in fact all that the proof uses about a Hamiltonian graph is that if we delete a vertex then the graph stays connected.
So the same proof would give the following.

\begin{thm}
 Let $\mathcal{F}$ be a family of graphs on $[n]$ such that the intersection of any two graphs in $\mathcal{F}$ has no cutvertex.
 Then $|\mathcal{F}|\le \frac{1}{2^n}\cdot 2^{\binom{n}{2}}$.
 \qed
\end{thm}

However, what if we only know that the intersections are bridgeless?
It ought to be the case that the best is again to take all graphs that contain a fixed $n$-cycle, but we have been unable to prove this.

\begin{conjecture}
 Let $\mathcal{F}$ be a family of graphs on $[n]$ such that the intersection of any two graphs in $\mathcal{F}$ is $2$-edge-connected (i.e. connected and bridgeless).
 Then $|\mathcal{F}|\le \frac{1}{2^n}\cdot 2^{\binom{n}{2}}$.
\end{conjecture}

In the `other direction', what if we weaken connectedness to having a bounded number of components?
This question is also raised in \cite{berger1}.
For two components, the natural example would be to take all graphs that contain all but one of the edges of a fixed $n$-cycle.
We strongly believe that this is the best example.

\begin{conjecture}
 Let $\mathcal{F}$ be a family of graphs on $[n]$ such that the intersection of any two graphs in $\mathcal{F}$ has at most two components.
 Then $|\mathcal{F}|\le \frac{n+1}{2^n}\cdot 2^{\binom{n}{2}}$.
\end{conjecture}

Note that, if correct, there are other extremal examples.
For example, we may replace the $n$-cycle by an $(n-1)$-cycle with a pendant edge.
Perhaps these are the only examples?
The only upper bound we have is based on projections onto double stars, and gives the very weak result that the size of $\mathcal{F}$ is at most $\frac{1}{2^{n/4}}$ of all graphs.

Interestingly, for more than two components, constructions based on a fixed Hamilton cycle no longer appear to be best.
For example, for three components we suspect that the best construction is as follows.
Let the graph $H$ consists of two half-sized cycles that meet at a point, and now take the family of all graphs that contain all of $H$ except for one edge of each cycle.
For clarity, we write this down precisely for $n$ odd.

\begin{conjecture}
 Let $\mathcal{F}$ be a family of graphs on $[n]$ ($n$ odd) such that the intersection of any two graphs in $\mathcal{F}$ has at most three components.
 Then $|\mathcal{F}|\le \frac{(n+3)^2}{8}\frac{1}{2^n}\cdot 2^{\binom{n}{2}}$.
\end{conjecture}

For intersections having at most $k$ components, it may be that the best construction (for $n$ sufficiently large) is to set $H$ to be $k-1$ cycles, of lengths as equal as possible, all meeting at one point, and then to take the family of all graphs that contain all but one edge of each cycle of $H$.
For example, for $k=4$, this gives a family whose size is $cn^3/2^n$ of all graphs, whereas the family of all graphs that contain all but two edges of a fixed $n$-cycle has size only $cn^2/2^n$ of all graphs.

\end{document}